\documentclass[11pt]{article}

\usepackage{amsmath,amsthm,amsfonts,amssymb, latexsym, eepic, floatflt, xypic, geometry,fancyhdr,cancel,color,picins}
\usepackage[dvips]{epsfig}
\newtheorem{thm}{Theorem}[section]

\newtheorem{lem}[thm]{Lemma}

\theoremstyle{definition}
\newtheorem*{defin}{Definition}

\DeclareMathOperator{\ric}{Ric}
\DeclareMathOperator{\spin}{Spin}
\DeclareMathOperator{\Cl}{Cl}

\def\rr{\mathbb{R}}

\def\too{\longrightarrow}

\DeclareMathOperator*{\ave}{ave}

\begin{document}

\title{On the near-equality case of the Positive Mass Theorem}
\author{
Dan A. Lee\\ Duke University\\ dalee@math.duke.edu 
}
\date{\today}
\maketitle

\begin{abstract}
The Positive Mass Conjecture states that any complete asymptotically flat manifold of nonnnegative scalar curvature has nonnegative mass. Moreover, the equality case of the Positive Mass Conjecture states that in the above situation, if the mass is zero, then the Riemannian manifold must be Euclidean space.  The Positive Mass Conjecture was proved by R. Schoen and S.-T. Yau for all manifolds of dimension less than $8$ \cite{schyau}, and it was proved by E. Witten for all spin manifolds \cite{witten}.  In this paper, we consider complete asymptotically flat manifolds of nonnegative scalar curvature that are also harmonically flat in an end.  We show that, whenever the Positive Mass Theorem holds, any appropriately normalized sequence of such manifolds whose masses converge to zero must have metrics that are uniformly converging to Euclidean metrics outside a compact region.  This result is an ingredient in a forthcoming proof, co-authored with H.\ Bray, of the Riemannian Penrose inequality in dimensions less than $8$ \cite{braylee}.
\end{abstract}

\section{Introduction}

The (Riemannian) Positive Mass Conjecture was proved by R. Schoen and S.-T. Yau for all manifolds of dimension less than $8$ using minimal hypersurfaces \cite{schyau} (see also \cite[Section 4]{montecatini}), and soon later it was proved by E. Witten for all spin manifolds using a Bochner-type argument \cite{witten} (see also \cite[Section 6]{bartnik}).  While the conjecture was originally motivated by general relativity, it has since proven to be fundamental for understanding scalar curvature.  Before describing our result, we first review the relevant background material.

\begin{defin}
Let $n\geq 3$.  A Riemannian manifold $(M^n,g)$ is said to be \emph{asymptotically flat}\footnote{Note that there are various inequivalent definitions of asymptotic flatness in the literature, but they are all similar in spirit.  This one is taken from \cite[Section 4]{montecatini}.} if there is a compact set $K\subset M$ such that $M\smallsetminus K$ is a disjoint union of \emph{ends}, $E_k$, such that each end is diffeomorphic to $\rr^n\smallsetminus B_1(0)$, and in each of these coordinate charts, the metric $g_{ij}$ satisfies
\begin{eqnarray*}
g_{ij}&=&\delta_{ij}+O(|x|^{-p})\\
g_{ij,k}&=&O(|x|^{-p-1})\\
g_{ij,kl}&=&O(|x|^{-p-2})\\
R_g&=&O(|x|^{-q})
\end{eqnarray*}
for some $p>(n-2)/2$ and some $q>n$, where the commas denote partial derivatives in the coordinate chart, and $R_g$ is the scalar curvature of $g$.

In this case, in each end $E_k$, the limit
$$m(E_k,g)={1\over 2(n-1)\omega_{n-1}}\lim_{\rho\to\infty}\int_{S_\rho} (g_{ij,i}-g_{ii,j})\nu_j d\mu$$
exists (see, e.g.\ \cite[Section 4]{montecatini}), where  $\omega_{n-1}$ is the area of the standard unit $(n-1)$-sphere, $S_{\rho}$ is the coordinate sphere in $E_k$ of radius $\rho$, $\nu$ is its outward unit normal, and $d\mu$ is the Euclidean area element on $S_{\rho}$.  We call the quantity $m(E_k,g)$, first considered by Arnowitt, Deser, and Misner (see, e.g.\ \cite{ADM}), the \emph{ADM mass} of the end $(E_k,g)$, or when the context is clear, we simply call it the \emph{mass}, $m(g)$.  (Under an additional assumption on the Ricci curvature, R.\ Bartnik showed that the ADM mass is a Riemannian invariant, independent of choice of asymptotically flat coordinates \cite{bartnik}.)

\end{defin}

\begin{defin}
If $M^n$ is a smooth manifold, we say that the \emph{Non-strict Positive Mass Theorem holds on $M$} if the following statement holds:  Given any complete asymptotically flat metric on $M$ with nonnegative scalar curvature, the mass of each end is nonnegative.\footnote{If $M$ has the wrong topology for supporting such metrics, then we may consider the statement to be vacuously true.}
\end{defin}
The Positive Mass Conjecture is the conjecture that the Non-strict Positive Mass Theorem holds on every smooth manifold.  We added the word ``non-strict'' because we have not yet described the equality case.  When $n=3$, this conjecture is natural in the context of general relativity; it simply says that if you have a time-symmetric slice of spacetime and the mass density is nonnegative, then the total mass must be nonnegative.  As mentioned above, this conjecture has been proven in many cases.
\begin{thm}[Non-strict Positive Mass Theorem]
If $M^n$ is spin or $n<8$, then the Non-strict Positive Mass Theorem holds on $M$.
\end{thm}
\begin{thm}[Equality case of Positive Mass Theorem]
Let $M^n$ be any smooth manifold on which the Non-strict Positive Mass Theorem holds.  If $(M,g)$ is a complete asymptotically flat manifold with nonnegative scalar curvature, and the mass of one of its ends is zero, then $(M,g)$ is isometric to Euclidean space.
\end{thm}
The reason we have described these theorems in this odd way is to emphasize that the equality case of the Positive Mass Theorem on $M$ always follows from the Non-strict Positive Mass Theorem on $M$.  The proof of this fact is due to Schoen and Yau \cite[Section 3]{schyau} and uses a variational argument.  The present work is a generalization of that argument.  In the spin case, one can also prove the equality case directly using Witten's spinor argument \cite{witten}.

A simple-minded view of the case of small mass is that if Euclidean space 
is the \emph{only} asymptotically flat manifold with nonnegative scalar 
curvature and \emph{zero} mass, then it stands to reason that any 
asymptotically flat manifold with nonnegative scalar curvature and 
\emph{small} mass should be \emph{close} to Euclidean space in some sense.  
Our 
goal is to see how far we can push this simple-minded idea.  One obvious 
problem here is that mass is not scale-invariant, and consequently, the 
hypothesis of small mass is not saying much.  Therefore, in addition to 
the hypothesis of small mass, we need to introduce another dimensional 
quantity in our hypotheses.  One way to do this is to assume that the 
metric is harmonically flat in an end (i.e.\ conformally flat in an end, 
with a harmonic conformal factor).  We now state the main result of this 
paper.

\begin{thm}[Main Theorem] Let $M^n$ be any smooth manifold on which the 
Non-Strict Positive Mass Theorem holds.  Let $\rr^n\smallsetminus 
B_{R}(0)$ be a coordinate chart for one of the ends of~$M$.  Let $g$ be a 
complete asymptotically flat metric of nonnegative scalar curvature on 
$M$, and suppose that $$g_{ij}(x)=U(x)^{4\over n-2}\delta_{ij}$$ in 
$\rr^n\smallsetminus B_{R}(0)$, where $U$ is a positive (Euclidean) 
harmonic function on $\rr^n\smallsetminus B_{R}(0)$ with 
$\lim_{x\to\infty} U(x)=1$.

Then for any $\epsilon>0$, there exists $\delta>0$ such that if the mass of $(M,g)$ in this end is less then $\delta R^{n-2}$, then 
$$\sup_{|x|>aR}|U(x)-1|<\epsilon,$$ 
where $a$ is a universal constant depending only on $n$.  The constant $\delta$ depends only on $\epsilon$ and $n$.  In particular, it does not depend on the topology of $M$.
\end{thm}

The role of mass is sometimes easier to understand in a harmonically flat end.  Taking $U(x)$ to be as in the theorem above and expanding it in spherical harmonics, we see that
\begin{equation}\label{expansion}
U(x)=1+{m(g)\over 2}|x|^{2-n} + O(|x|^{1-n}).
\end{equation}
The hypothesis of a harmonically flat end may seem restrictive, but it is 
sometimes a reasonable one for geometric applications.  See \cite[Section 
4]{montecatini}, for example.
More importantly, the original motivation for this work was for use in the proof of the Riemannian Penrose inequality in dimensions less than $8$, and for this purpose, the above theorem suffices.  See \cite{braylee}.  Still, it would be interesting to understand how to relax the hypotheses.

While the theorem was originally contemplated with the idea that $R$ would be fixed while the mass approaches zero, the scale invariance of the theorem also gives us a result when the mass is fixed and $R$ approaches infinity.  If we also fix $g$, we recover the uninteresting fact that $g$ becomes flatter as $x\to\infty$.  However, by considering a family of metrics with uniformly bounded mass, the theorem tells us that these metrics become flatter as $x\to\infty$ \emph{uniformly in $g$}.

Also, we note that if $M$ is spin, then our main result follows from Witten's spinor proof of the Positive Mass Theorem.  This argument is given in \cite[Section 12]{bray}, but we include it here in Appendix A for the sake of completeness.  Another related result is the work of H.\ Bray and F.\ Finster \cite{brayfinster}, which also assumes that $M$ is spin.

The author thanks H.\ Bray for several helpful conversations, and also
 R.\ Schoen for 
suggesting how the variational argument used by Schoen and Yau for the 
equality case of the Positive Mass Theorem could be adapted for this work.

\section{Proof of Main Theorem}

Because of the scale invariance of the theorem, we may assume, without loss of generality, that $R=1$ for the remainder of this paper.

We will first prove the theorem under the assumption that $g$ is scalar-flat everywhere.  We will prove the theorem using a variational argument.  For now, let $a$ be some fixed constant larger than $3$.  Later, we will see how large $a$ needs to be in order for the theorem to hold.  Let $\varphi$ be a smooth cutoff function such that 
$$\varphi(x)=\left\{
\begin{array}{lll}
0&\text{ for }&|x|\leq a/3\\
1&\text{ for }&a/2\leq |x|\leq 3a\\
0&\text{ for }&|x|\geq 4a.\\
\end{array}
\right.$$
For a real parameter $s$, define $g_s=g+s\varphi\ric(g)$.  For small $|s|$, we can find a conformal factor $u_s$ such that $\lim_{x\to\infty}u_s(x)=1$ and $u_s^{4\over n-2}g_s$ is scalar-flat everywhere. Define \hbox{$m(s)=m\left(u_s^{4\over n-2}g_s\right)$}.  Note that this construction explicitly depends on the constant $a$.  Recall that scalar-flatness of $u_s^{4\over n-2}g_s$ means that $u_s$ solves the conformal Laplace's equation:
$$\Delta_{s}u_s-{n-2\over 4(n-1)}R_{s} u_s=0,$$
where $\Delta_s$ and $R_s$ denote the Laplacian and scalar curvature with respect to $g_s$.  Throughout this paper, $s$ subscripts will follow this convention, except for the subscript in $u_s$.  In particular, a $0$-subscript refers to the metric $g=g_0$.  The absence of a subscript means that the Euclidean metric is being used.  We will use the ``dot'' notation for derivatives with respect to the variable $s$, and we also adopt the shorthand notation $B_{\rho}=B_{\rho}(0)$ and $S_{\rho}=S_{\rho}(0)$.

\begin{lem}
Let $(M,g)$ be a scalar-flat, complete asymptotically flat manifold, and suppose that in one of the ends, $\rr^n\smallsetminus B_{1}$, we have
$$g_{ij}(x)=U(x)^{4\over n-2}\delta_{ij}$$
 where $U$ is a positive harmonic function on $\rr^n\smallsetminus B_{1}$ 
with $\lim_{x\to\infty} U(x)=1$. Given any $a>3$, and then defining $m(s)$ 
as in the paragraph above, we have $$\sup_{|x|>a}|U(x)-1|\leq 
C\left[(a^{4-n}\dot{m}(0))^{1\over4}+a^{2-n}|m(0)|\right],$$ for some 
constant $C$ depending only on $n$. \end{lem} \begin{proof} Examining the 
expansions of $U$ and $u_s U$ in spherical harmonics, and using 
equation~(\ref{expansion}), we find that \begin{eqnarray*} 
m(s)-m(0)&=&{-2\over \omega_{n-1}(n-2)}\lim_{\rho\to\infty} \int_{S_\rho} 
u_s {\partial u_s\over\partial r}\,d\mu\\ &=& {-2\over 
\omega_{n-1}(n-2)}\lim_{\rho\to\infty} \int_{S_\rho} u_s \nabla_s u_s\cdot 
\nu_s\,d\mu_s\\ &=& {-2\over \omega_{n-1}(n-2)} \int_{M}( |\nabla_s 
u_s|_s^2 + u_s \Delta_s u_s) \,d\mu_s\\ &=& {-2\over \omega_{n-1}(n-2)} 
\int_{M}( |\nabla_s u_s|_s^2 + {n-2\over 4(n-1)} R_s u_s^2) \,d\mu_s.\\ 
\end{eqnarray*} 
Noting that $u_0=1$ and $g_0$ is scalar-flat, we compute 
\begin{eqnarray*} \dot{m}(0)&=& {-1\over 2\omega_{n-1}(n-1)} \int_{M} 
\dot{R}_0 \,d\mu_0\\ &=& {1\over 2\omega_{n-1}(n-1)} \int_{M} \langle 
\dot{g}_0, \ric_0\rangle_0 \,d\mu_0\\ &=& {1\over 2\omega_{n-1}(n-1)} 
\int_{M} \varphi|\ric_0|_0^2 \,d\mu_0\\ &\geq& {1\over 2\omega_{n-1}(n-1)} 
\int_{B_{3a}-B_{a/2}} |\ric_0|_0^2 \,d\mu_0,\\ \end{eqnarray*} where the 
other terms in $\dot{R}_0$ are divergences. (Differentiation under the 
integral can be justified as in \cite[Section 3]{schyau}.)  In particular, 
we see that $\dot{m}(0)\geq0$.
Analyzing the 
formula for $\ric(U^{4\over n-2}\delta_{ij})$, we find that 
$$\int_{B_{3a}-B_{a/2}} |\nabla U|^4\,d\mu\leq C \dot{m}(0),$$ for some 
constant $C$ depending only on $n$.  Since $U$ is harmonic, it follows 
that $$\sup_{a<|x|<2a}\left|U(x)-\ave_{a<|y|<2a} U(y)\right|\leq 
C\left(a^{4-n}\int_{B_{3a}-B_{a/2}} |\nabla U|^4\,d\mu\right)^{1\over 
4},$$ 
for some $C$ depending only on $n$.  Using the maximum principle and the 
fact that the spherical average of $U$ on $S_r$ is 
$1+{m(0)\over2}r^{2-n}$, we find that 
\begin{eqnarray*} 
\sup_{|x|>a}|U(x)-1|&\leq& \sup_{a<|x|<2a}|U(x)-1|\\ &\leq& 
C(a^{4-n}\dot{m}(0))^{1\over4}+\left|\ave_{a<|x|<2a}U(x)-1\right| \\ &=& 
C(a^{4-n}\dot{m}(0))^{1\over4}+ {1\over|B_{2a}-B_{a}|} 
\left|\int_{a}^{2a}\omega_{n-1}r^{n-1} \left({m(0)\over 2}r^{2-n}\right) 
\,dr\right|\\ &\leq& 
C\left[(a^{4-n}\dot{m}(0))^{1\over4}+a^{2-n}|m(0)|\right], \end{eqnarray*} 
where we may have had to enlarge the constant at the last step. 
\end{proof}

\begin{defin}
Given $n$, let $\mathcal{A}$ be the set of all scalar-flat, complete asymptotically flat manifolds $(M,g)$, such that in one of the ends, $g_{ij}(x)=U(x)^{4\over n-2}\delta_{ij}$ in  $\rr^n\smallsetminus B_{1}$, where $U$ is a positive harmonic function on $\rr^n\smallsetminus B_{1}$ with $\lim_{x\to\infty} U(x)=1$, and $|m(g)|\leq 1$.
\end{defin}

\begin{lem}\label{biglemma}
There exist positive constants $a$, $s_0$, and $C_0$, depending only on $n$, such that for all $(M,g)\in\mathcal{A}$ and $|s|<s_0$, the conformal factor $u_s$ described above exists, and $|\ddot{m}(s)|<C_0$.
\end{lem}

Before proving this lemma, let us see how the two lemmas together prove 
our Main Theorem in the scalar-flat case.  Once a value of $a$ is fixed, 
the first lemma (together with the Positive Mass Theorem) reduces the 
problem to showing that for any $\gamma>0$, 
there exists $\delta>0$ such that $m(0)<\delta$ implies that 
$\dot{m}(0)<\gamma$.   Let $C_0$ and $s_0$ be the constants described in the 
previous lemma, and without loss of generality, assume that $\gamma<C_0 
s_0$, so that $|-\gamma/C_0|<s_0$.  Let $m(0)<\delta ={\gamma^2\over 
2C_0}$.  Assuming that the Non-strict Positive Mass Theorem holds on $M$, 
we have $m(-\gamma/ C_0)\geq0$.  Therefore

\begin{eqnarray*}
{\gamma^2\over 2C_0}-0&>&m(0)-m(-\gamma/C_0)\\
&=&\int_{-\gamma/C_0}^0 \dot{m}(s)\,ds\\
&\geq&\int_{-\gamma/C_0}^0 (\dot{m}(0)+C_0 s)\,ds\\
&=&\dot{m}(0){\gamma\over C_0}-{\gamma^2\over2C_0}.
\end{eqnarray*}
Thus $\dot{m}(0)<\gamma$.  The argument given above is the only place in this paper where we invoke the assumption that the Non-strict Positive Mass Theorem holds on $M$.  Specifically, it is not assumed in the statements of any of the lemmas.

The following lemma is the fundamental reason why Lemma \ref{biglemma} is true. 
\begin{lem}\label{uniform}
There exist constants $C_1$ and $C(k)$, depending only on $n$, such that for all $(M,g)\in\mathcal{A}$, $|x|\geq 2$, 
$${1\over C_1} < U(x) < C_1$$
$$|U(x)-1|\leq C(0)|x|^{2-n}$$
$$|\nabla^k U(x)|\leq C(k)|x|^{2-n-k}$$
for each positive integer $k$.
\end{lem}
\begin{proof}
By the Harnack inequality, we know that there is a constant $C$ such that, for all $|x|\geq 3/2$,
$${1\over C}\ave_{S_{|x|}}U < U(x) < C\ave_{S_{|x|}} U.$$
Since the average value of $U$ on $S_r$ is $1+{m(g)\over 2}r^{2-n}$, and $|m(g)|\leq 1$, we conclude that
$${1\over 2C} <  {1\over C}\left(1-{1\over 2}|x|^{2-n}\right)< U(x) < 
C\left(1+{1\over 2}|x|^{2-n}\right)<{3C\over2}$$
for all $|x|>3/2$.  This establishes the first inequality.  Since we also know that \hbox{$\lim_{x\to\infty}U(x)=1$}, it follows from the maximum principle that for some new constant $C$ and all $|x|\geq 3/2$,
$$1 - C |x|^{2-n}<U(x)< 1+C|x|^{2-n}.$$
Since $U$ is harmonic, the gradient bounds follow routinely.
\end{proof}

The previous lemma gives us uniform estimates on all $(M,g)\in\mathcal{A}$ and all of their derivatives for $|x|>2$.  It is clear that there exist constants $s_0$ and $C_0$ as in the statement of the lemma, \emph{if} we allow them to depend on $g$.  Therefore uniform estimates on $g$ should give us the result we want.  The only issue here is that we do not have uniform estimates on $(M,g)\in\mathcal{A}$ away from the region $|x|>2$, so we must argue that the behavior of $g$ there does not cause a problem.  But this is reasonable to expect since the mass is computed at infinity and $\varphi$ is supported in the region $|x|>a/3$.

Consider $(M,g)\in\mathcal{A}$.  For any nonnegative integer $k$ and $0<\alpha<1$, let $C^{k,\alpha}_\sigma(M)$ be the weighted Schauder space\footnote{See \cite{locmco} for background on weighted \emph{Sobolev} spaces.  While the corresponding theory for weighted \emph{Schauder} spaces (also known as H\"{o}lder spaces) is well-understood, the literature is somewhat sparse.  See \cite[Section 6.1]{marshall}.} of $C^{k,\alpha}$ functions on $M$ whose $m$-th derivatives are $O(|x|^{\sigma-m})$ for all $m\leq k$, and whose ``$k+\alpha$ H\"{o}lder norms are $O(|x|^{\sigma-k-\alpha})$.''  There are many equivalent ways to define a weighted Schauder norm on the space $C^{k,\alpha}_\sigma(M)$.  We choose one with the property that for a function $v$ supported in $|x|>1$, the norm is equal to the usual one for Euclidean space.  In other words, for $v$ supported in $|x|>1$,
\begin{equation}\label{norm}
|v|_{C^{k,\alpha}_\sigma(M)}=\sup \left||x|^{-\sigma}v\right|+\sup \left||x|^{1-\sigma}\nabla
v\right|+\cdots+\sup\left||x|^{k-\sigma}\nabla^k v\right|+\left[|x|^{k+\alpha-\sigma}\nabla^k
v\right]_{\alpha},
\end{equation}
where $\nabla$ is the flat connection on $\rr^n$.  The basic point is that we are choosing a weighted Schauder norm that is independent of $(M,g)\in\mathcal{A}$ in the region $|x|>1$, and it does not matter what it is elsewhere. 
\begin{defin}
We adopt the shorthand notation notation that for $\rho>1$, $|v|_{k+\alpha,\sigma,\rho}$ is the (Euclidean) $C^{k,\alpha}_\sigma(\rr^n\smallsetminus B_{\rho})$ norm of $v$.  In other words, it is the norm described in equation~(\ref{norm}), restricted to $\rr^n\smallsetminus B_{\rho}$.
\end{defin}

For the rest of this paper, we fix a specific value of $\alpha\in(0,1)$ and a specific value of $\sigma\in(2-n,0)$.

Let $L_s$ be the conformal Laplacian for the metric $g_s$, that is,
$$L_s u=\Delta_s u -{4(n-1)\over n-2}R_s u.$$
Recall that we are looking for $u_s$ such that $\lim_{x\to\infty}u_s(x)=1$ and $L_s u_s=0$.  Equivalently, we seek $v_s=u_s-1\in C^{2,\alpha}_\sigma(M)$ such that $L_s v_s ={4(n-1)\over n-2}R_s$.  
It is well-known that for  $\sigma\in(2-n,0)$,
$$\Delta_g:C^{2,\alpha}_\sigma(M)\too C^{0,\alpha}_{\sigma-2}(M)$$ 
is an isomorphism.  Then since
$$L_s:C^{2,\alpha}_\sigma(M)\too C^{0,\alpha}_{\sigma-2}(M)$$
is a deformation of $\Delta_g$, it is natural to use an Inverse Function Theorem argument to prove that $L_s$ is also an isomorphism for small $s$.  However, since we do not have global uniform control over $g$, we have to modify this basic argument.

\begin{defin}
Let $\mathcal{B}$ be the space of all functions in $C^{2,\alpha}_\sigma(M)$ that are $\Delta_g$-harmonic on $M$ away from the region $\rr^n\smallsetminus B_{a/3}(0).$ 
\end{defin}

\begin{lem}
For a large enough value of $a>10$, there exists $C_3$ depending only on $\sigma$, $\alpha$, and $n$, such that for all $(M,g)\in\mathcal{A}$ and all $v\in\mathcal{B}$,
$$|v|_{2+\alpha,\sigma,a/4}< C_3 |\Delta_g v|_{\alpha,\sigma-2,a/4}.$$
\end{lem}

\begin{proof}
Specfically, we choose $a$ large enough so that $(a/10)^{\sigma}< {1\over C_1^2}$, where $C_1$ is the constant from Lemma \ref{uniform}. We define a function $f_\infty$ on $|x|\geq 2$ by the formula 
$$f_\infty(x)= -{|x|^{\sigma}\over U(x)}.$$
Meanwhile, we define $f_0$ in the complement of $|x|>a/5$ to be the unique $\Delta_g$-harmonic function with boundary values $-{1\over C_1}2^{\sigma}$ at $|x|=a/5$, and zero at the infinities of the other ends (if there are any).  Finally, we define a function $f$ on all of $M$ by setting 
$$f(x)=
\left\{\begin{array}{ll}
f_\infty(x) & \text{ for }|x|\geq a/5\\
\max(f_0(x),f_\infty(x))& \text{ for }2\leq x\leq a/5\\
f_0(x) & \text{ elsewhere.}
\end{array}\right.$$
Observe that at $|x|=2$, by Lemma \ref{uniform}, 
$$f_\infty(x)\leq -{2^\sigma\over C_1}\leq f_0(x),$$
by the maximum principle.
On the other hand, at $|x|= a/5$, 
$$f_\infty(x)\geq -{C_1 (a/5)^\sigma} \geq -{1\over C_1}2^{\sigma}=f_0(x),$$
by our assumption on $a$.  In other words, we see that $f$ is a continuous function on $M$, and that the locus where $f_0(x)=f_\infty(x)$ is confined  to the region $2<x<a/5$.

Now it is a straightforward computation to see that for $|x|>2$,
we have 
$$\Delta_g f_\infty(x) = -\sigma(n-2+\sigma)|x|^{\sigma-2}  U(x)^{-{n+2\over n-2}}.$$  Since $\sigma\in(2-n,0)$ and we have the uniform upper bound $U(x)<C_1$, we see that for all $|x|>2$,
$$\Delta_g f_\infty(x)> {1\over C_2}|x|^{\sigma-2},$$
for some constant $C_2$.  We now see that $f$ is the maximum of two $\Delta_g$-subharmonic functions, and hence $f$ is globally $\Delta_g$-subharmonic.  Directly from the definitions, we know that for all $v\in\mathcal{B}$, $|x|>1$,
$$\left|{\Delta_g v(x)\over |\Delta_g v|_{0,\sigma-2, a/5}}\right|\leq |x|^{\sigma-2}.$$
Therefore, for $|x|>a/5$,
$$\Delta_g\left( { v\over |\Delta_g v|_{0,\sigma-2, a/5}}+C_2 f\right)>0.$$
Also, away from $|x|>a/3$, we know that $\Delta_g v(x)=0$ and that $f$ is $\Delta_g$-subharmonic.  Consequently, for any $v\in\mathcal{B}$,  ${ v\over |\Delta_g v|_{0,\sigma-2, a/5}}+C_2 f$ is a globally $\Delta_g$-subharmonic function on~$M$ that approaches zero at the infinities of each end.  By the maximum principle, we see that for all $x\in M$,
$${v(x)\over |\Delta_g v|_{0,\sigma-2, a/5}}+C_2 f(x)\leq 0.$$
Specifically, for $|x|>a/5$,
$${v(x)\over |\Delta_g v|_{0,\sigma-2, a/5}}\leq C_2{|x|^\sigma\over U(x)}.$$
Since we also have the uniform lower bound $U(x)>{1\over C_1}$ (Lemma \ref{uniform}), we obtain for $|x|>a/5$,
$$|x|^{-\sigma}v(x) \leq C_1 C_2 |\Delta_g v|_{0,\sigma-2, a/5}.$$
A similar argument applied to the $\Delta_g$-superharmonic function ${v(x)\over |\Delta_g v|_{0,\sigma-2, a/5}}-C_2 f(x)$ then gives us the following important estimate.
\begin{equation}\label{injective}
|v|_{0,\sigma,a/5}  \leq C_1 C_2|\Delta_g v|_{0,\sigma-2, a/5}.
\end{equation}

The standard weighted elliptic estimate tells us that for any $v\in C^{2,\alpha}_\sigma(M)$,
$$|v|_{2+\alpha,\sigma,a/4}<C\left(|\Delta_g v|_{\alpha,\sigma-2,a/5}+|v|_{0,\sigma,a/5}\right),$$
where $C$ is independent of $(M,g)\in\mathcal{A}$ because of our uniform bounds from Lemma \ref{uniform}.  Combining this with inequality (\ref{injective}),  we see that for $v\in\mathcal{B}$,
$$|v|_{2+\alpha,\sigma,a/4}< C_3 |\Delta_g v|_{\alpha,\sigma-2,a/4}$$
for some new constant $C_3$.
\end{proof}

\begin{proof}[Proof of Lemma \ref{biglemma}]

Again, using our uniform control over $(M,g)\in\mathcal{A}$ in the region $|x|>2$ (Lemma \ref{uniform}), it is clear that we can choose $s_0>0$ such that for all $|s|<s_0$ and $v\in\mathcal{B}$, 
$$|(L_s-\Delta_g)v|_{\alpha,\sigma-2,a/4}\leq {1\over 2C_3}|v|_{2+\alpha,\sigma,a/4},$$
where $C_3$ is the constant from the previous lemma.  Combining this estimate with the previous lemma, we find that for all $|s|<s_0$ and all $v\in\mathcal{B}$,
\begin{equation}\label{estimate}
|v|_{2+\alpha,\sigma,a/4}< 2C_3 |L_s v|_{\alpha,\sigma-2,a/4}.
\end{equation}
In particular, we observe that the restriction of $L_s$ to $\mathcal{B}$ is injective for $|s|<s_0$.  Since $L_s$ is equal to $\Delta_g$ away from $|x|>a/3$, it follows that kernel of $L_s$ in $C^{2,\alpha}_\sigma(M)$ is contained in $\mathcal{B}$.  Consequently, the operator $L_s:C^{2,\alpha}_\sigma(M)\too C^{0,\alpha}_{\sigma-2}(M)$ is injective.  Since it is also Fredholm of index zero, we see that $L_s$ is an isomorphism.  This establishes that the desired $v_s=u_s-1$ exists for all $|s|<s_0$.

Next, we need to show that there exists a $C_0$ such that $|\ddot{m}(s)|<C_0$.  By the Inverse Function Theorem, we know that $v_s=u_s-1$ depends smoothly on $s$.  Going back to the equation
$$m(s)-m(0)={-2\over \omega_{n-1}(n-2)}\lim_{\rho\to\infty} \int_{S_\rho} u_s {\partial u_s\over\partial r}\,d\mu,$$
we see that 
$$\ddot{m}(s)={-2\over \omega_{n-1}(n-2)}\lim_{\rho\to\infty} \int_{S_\rho} {\partial \ddot{v}_s\over\partial r}\,d\mu.$$
Therefore, all we need is a uniform bound on $|\ddot{v}_s|_{2+\alpha,\sigma,a/4}$.  But this follows easily from a bootstrapping argument from the bound (\ref{estimate}).  Specifically, we find that the bound on $|\ddot{v}_s|_{2+\alpha,\sigma,a/4}$ depends on the constant $C_3$, $|R_s |_{\alpha,\sigma-2,a/4}$, $|\dot{R}_s |_{\alpha,\sigma-2,a/4}$, $|\ddot{R}_s |_{\alpha,\sigma-2,a/4}$, and the norms of $L_s$, $\dot{L}_s$, and $\ddot{L}_s$ as maps from $C^{2,\alpha}_\sigma(\rr^n\smallsetminus B_{a/4})$ to $C^{0,\alpha}_{\sigma-2}(\rr^n\smallsetminus B_{a/4})$, all of which are controlled.
\end{proof}

We have now completed the proof for the scalar-flat case.  Now let us 
consider the general case of our Main Theorem, which considers metrics of 
nonnegative scalar curvature. 

\begin{lem}
Let $(M, g)$ be a complete asymptotically flat manifold of 
nonnegative
scalar curvature.  Let $\rr^n\smallsetminus B_1$ be a coordinate chart
for one of the ends of $M$, and suppose that
$$g_{ij}(x)=U(x)^{4\over n-2}\delta_{ij}$$  
in  $\rr^n\smallsetminus B_1$, where
$U$ is a positive harmonic function on $\rr^n\smallsetminus B_{1}$ with $\lim_{x\to\infty} U(x)=1$.

Then there exists a scalar-flat asymptotically flat metric $\tilde{g}$ on $M$ and a positive function $v$ on $M$, such that $v$ approaches $1$ at the infinities of all the ends and $g=v^{4\over n-2}\tilde{g}$.  Setting $\tilde{U}=U/v$ and choosing a constant $a>1$,
$$0\leq\sup_{|x|>a} (U(x)-\tilde{U}(x)) < C (m(g)-m(\tilde{g})),$$
for some constant $C$ depending only on $a$ and $n$.
\end{lem}
This lemma, combined with an application of the scalar-flat case of the Main Theorem to $(\tilde{g},\tilde{U})$, clearly proves the general case of the Main Theorem.

\begin{proof} First, the existence of $\tilde{g}$ and $v$ comes 
immediately by solving the conformal Laplace's equation.  In particular, 
we see that $v$ is a $\tilde{g}$-superharmonic function.  Setting 
$\tilde{U}=U/v$ as above, we see that both $U$ and $\tilde{U}$ are 
Euclidean harmonic functions.  Since $v$ is $\tilde{g}$-superharmonic 
and 
approaches $1$ at the infinities of the ends, the maximum principle tells 
us that that $v(x)-1$ is a nonnegative function. Consequently, 
$U-\tilde{U}=\tilde{U}(v-1)$ is a nonnegative Euclidean harmonic 
function on 
$|x|>1$.

Let $K(x,y)$ be the Poisson kernel for the Euclidean Laplacian on $\rr^n\smallsetminus B_{(a+1)/2}$ (assuming zero boundary condition at infinity).  Then we have, for any $|x|>a$,
$$U(x)-\tilde{U}(x)=\int_{S_{(a+1)/2}}K(x,y)(U(y)-\tilde{U}(y))\,d\mu_y.$$
Since $U-\tilde{U}$ is nonnegative, we find that
\begin{eqnarray*}
\sup_{|x|>a} (U(x)-\tilde{U}(x)) &\leq& \int_{S_{(a+1)/2}} \left(\sup_{|x|>a,|\xi|=(a+1)/2} K(x,\xi)\right) (U(y)-\tilde{U}(y))\,d\mu_y\\
&=& C(m(g)-m(\tilde{g})),
\end{eqnarray*}
for some $C$ depending only on $a$ and $n$.  In the final step, we used 
the fact that the average values
of $U$ and $\tilde{U}$ on $S_{(a+1)/2}$ are $1+{m(g)\over 
2}\left({a+1\over 2}\right)^{2-n}$ and  $1+{m(\tilde{g})\over 
2}\left({a+1\over 2}\right)^{2-n}$, respectively.

\end{proof}

\appendix
\section{The spin case}

If we assume that our manifold is spin, then there is an easier way to 
prove our main theorem.  The statement also becomes slightly stronger.  
This argument appears in \cite[Section 12]{bray}.  For background on 
Witten's spinor argument, see \cite{bartnik}. 
\begin{thm}\label{spin} Let 
$(M^n, g)$ be a complete asymptotically flat spin manifold of nonnegative 
scalar curvature. Let $\rr^n\smallsetminus B_{R}(0)$ be a coordinate chart 
for one of the ends of $M$, and suppose that
$$g_{ij}(x)=U(x)^{4\over n-2}\delta_{ij}$$
in  $\rr^n\smallsetminus B_{R}(0)$, where
$U$ is a positive (Euclidean) harmonic function on $\rr^n\smallsetminus B_{R}(0)$ with $\lim_{x\to\infty} U(x)=1$.  

Then for any $a>1$, $\epsilon>0$, there exists $\delta>0$ such that if the mass of $(M,g)$ in this end is less then $\delta R^{n-2}$, then 
$$\sup_{|x|>aR}|U(x)-1|<\epsilon.$$ 
The constant $\delta$ depends only on $\epsilon$, $a$, and $n$.  In particular, it does not depend on the topology of $M$.
\end{thm}

\begin{proof}[Sketch of proof]
Again, without loss of generality, we assume $R=1$.  Given $n$, define 
$\mathcal{A}^{\spin}$ to be the set of all complete 
asymptotically flat spin manifolds $(M,g)$ of nonnegative scalar curvature, such that in one of the ends, $g_{ij}(x)=U(x)^{4\over n-2}\delta_{ij}$ in  $\rr^n\smallsetminus B_{1}$, where $U$ is a positive harmonic function on $\rr^n\smallsetminus B_{1}$ with $\lim_{x\to\infty} U(x)=1$, and $|m(g)|\leq 1$.

Define $H$ to be the set of all positive harmonic functions $U$ on $\rr^n\smallsetminus B_{a}$ such that $\lim_{x\to\infty}U(x)=1$, where the topology on $H$ is given by the $\sup$ metric.

Let $S$ be a Dirac spinor bundle over $M$.  (In brief: Since $M$ is spin, we can lift the orthonormal frame bundle to a principal $\spin(n)$-bundle over $M$.  Choose an irreducible representation of the Clifford algebra, $\Cl(n)$.  This representation restricts to a representation of $\spin(n)$, and we define $S$ to be the bundle associated to it.)  Fix a basis of constant spinors, $\psi_i$, of norm $1$, and define functionals $\mathcal{F}_i:H\too \rr$ by the formula
$$\mathcal{F}_i(U)=\inf \left\{\left. \int_{\rr^n\smallsetminus B_{2a}}|\nabla \psi|^2\,dV\right| \psi\in\Gamma(\rr^n\smallsetminus B_{2a},S)\text{ such that }\lim_{x\to\infty}\psi(x)=\psi_i\right\},$$
where $|\nabla\psi|^2$ is computed using the spin connection and bundle metric induced by the metric $g_{ij}(x)=U(x)^{4\over n-2}\delta_{ij}$.

We now make the following observations:
\begin{itemize}
\item For each $i$, the functional $\mathcal{F}_i$ is continuous on 
$H$.  
(This is 
because the functional is always minimized and the unique minimizer 
depends continuously on $U$.)
\item For all $U\in H$, if $\mathcal{F}_i(U)=0$ for each $i$, then $U$ is the constant function $1$.  (This follows from the fact that a basis of parallel spinors implies a flat metric.)
\item The restriction map from $\mathcal{A}^{\spin}$ to $H$ has relatively compact image in $H$.  (This follows from the same reasoning as in Lemma \ref{uniform}).
\item For any $(M,g)\in\mathcal{A}^{\spin}$ with corresponding 
harmonic function $U$, we know that for each $i$,  $m(g)\geq 
c(n)\mathcal{F}_i(U)$ 
for some constant $c(n)$.  (This follows directly from Witten's spinor argument \cite{witten}.)
\end{itemize}
It is an elementary exercise to combine the four bullet points above to prove the theorem.
\end{proof}

\bibliographystyle{alpha}
\bibliography{research07}

\end{document}